\documentclass[12pt]{article}
\usepackage{amsmath}
  \usepackage{paralist}
  \usepackage{graphics} 
  \usepackage{epsfig} 
\usepackage{graphicx}  \usepackage{epstopdf}
 \usepackage[colorlinks=true]{hyperref}
\usepackage{amsmath,amsthm,amsfonts}
\hypersetup{urlcolor=blue, citecolor=red}

\newtheorem{theorem}{Theorem}[section]

\newtheorem{lemma}[theorem]{Lemma}
\newtheorem{proposition}{Proposition}

\theoremstyle{definition}
\newtheorem{definition}[theorem]{Definition}
\newtheorem{remark}{Remark}

\def\c#1{\mathop{\times}\limits^{#1}}
\def\fl#1{\left\lfloor#1\right\rfloor}
\newcommand{\Real}{\mathbb{R}}
\newcommand{\Integer}{\mathbb{Z}}

\def\Z{{\mathbb Z}}
\def\N{{\mathbb N}}
\def\R{\mathbb{R}}
\def\C{\mathbb{C}}
\def\Mat{\mathrm{Mat}}
\def\det{\mathrm{det}}
\def\rank{\mathrm{rank}}
\def\arg{\mathrm{arg}}

\def\cH{\mathcal{H}}
\def\cA{\mathcal{A}}

\title{Directional complexity and entropy for lift mappings}

\author{Afraimovich, Courbage and Glebsky}

\begin{document}
\maketitle

\centerline{\scshape  V. Afraimovich}
\medskip
{\footnotesize
 \centerline{ Instituto de Investigaci\'on en Comunicaci\'on \'Optica, Universidad Aut\'onoma de San Luis Potos\'{\i},}
   \centerline{Karakorum 1470, Lomas 4a 78220, San Luis Potosi, S.L.P, Mexico}
} 

\medskip

\centerline{\scshape  M. Courbage}
\medskip
{\footnotesize
 \centerline{Laboratoire Mati\`ere et Syst\`emes Complexes (MSC),  UMR 7057
CNRS  et  Universit\'e Paris 7-Denis Diderot}
   \centerline{10, rue Alice Domon et L\'eonie Duquet 75205 Paris Cedex
13, France}
}

\medskip

\centerline{\scshape  L. Glebsky}
\medskip
{\footnotesize
 \centerline{ Instituto de Investigaci\'on en Comunicaci\'on \'Optica, Universidad Aut\'onoma de San Luis Potos\'{\i}}
   \centerline{Karakorum 1470, Lomas 4a 78220, San Luis Potosi, S.L.P, Mexico}
}

\bigskip


\begin{abstract}
We introduce and study the notion of a directional complexity and entropy for maps of degree $1$
on the circle. For piecewise affine Markov maps we use symbolic dynamics to relate this complexity to the symbolic complexity. We apply a combinatorial machinery to obtain exact 
formulas for the directional entropy, to find the maximal directional entropy, and to show that 
it equals the topological entropy of the map. 
AMS classification 37E10, 37E45.
\end{abstract}

\section*{Introduction}

There is a well-developed theory of rotation vectors (numbers) and rotation sets (see, for instance,
\cite{GM} and reference therein). One considers a map $f:M\to M$ generating a dynamical system and an
observable $\phi:M\to \R^d$ that classically is a displacement but might be an arbitrary function. The
rotation vector of $x$ is the Birkhoff average
$$
\lim_{n\to\infty}\frac{1}{n}\sum_{i=0}^{n-1}\phi(f^ix),
$$  
provided that the limit exists, say, equals $v$. Then we may say that $x$ moves in the direction $v$. 
A natural question arises: how many points move in the direction $v$ if one measures them in terms of the 
topological entropy. The authors of \cite{GM} have mentioned several attempts to answer the question
and have described their own approach. All of them including one of \cite{Kw} are based on the thermodynamic
formalism, in particular, on the variational principle. In our article we use purely topological (metric)
approach to describe points moving to the prescribed direction.   

 We shall exploit notion of the $\epsilon$-separability
introduced by Kolmogorov and Tikhomirov \cite{KT} in the context of
\cite{AZ}. A notion of space-time window introduced in \cite{M86,M88} for
cellular automata and used in \cite{ACFM,AMU,CK} for lattice
dynamical systems we apply here for maps on $\Real^1$ that are lifts
for maps of the circle of degree $1$. If such a map generates the dynamical system
with non-zero topological entropy  then, very often, it has a rotation interval different
from a single point. It implies the existence of trajectories with different rotation numbers,
i.e. with different spatio-temporal features. We suggest here to measure the number of 
trajectories with a given rotation number using the notion of a directional 
entropy. Roughly speaking if $X$ is a subset of a the circle such that the trajectories
going through $X$ have the rotation number, say, $\alpha$, then the 
$(\epsilon,n)$-complexity
of $X$ behaves asymptotically ($n>>1$) as $\exp(n\cH_\alpha)$. 
We call the number $\cH_\alpha$ the 
directional entropy in the direction $\alpha$.\footnote{The term directional complexity was used in \cite{GR} in another context. In \cite{GR} the direction is the physical 
direction in billiards} The greater $\cH_\alpha$ the greater the rate of
instability manifests by trajectories with the rotation number $\alpha$. 
But one has to be careful. It can happen (and occurs for mixing systems) that for any fixed
rotation number $\alpha$ inside the rotation interval the set of initial points, say 
$X_\alpha$, corresponding to this rotation number is dense in the circle. So, the 
topological entropy on 
$X_\alpha$ coincide with the  topological entropy of the whole system. To avoid it we approximate
$X_\alpha$ by sets of initial points which trajectories stay in a space-time window, calculate
the entropy on this window, and obtain $\cH_\alpha$ as the limit of these entropies.   

In this article we study mainly piecewise affine Markov maps of the circle. For such maps
it is possible to replace the calculation of the $(\epsilon,n)$-complexity by that of the 
symbolic complexity of some subsets of a corresponding topological Markov chain (TMC).
The TMC is determined by the Markov partition of the circle and the subsets -- by the admissibility 
condition formulated according to the value of the rotation number. After that the problem
becomes purely combinatorial. We use the approach of \cite{P1,P2} adjusted for our situation to
obtain the explicit formulas for $H_\alpha$. The formulas depend only on the entries of the transition
matrix of the TMC and on the weights of the edges of the corresponding oriented graph, where
the weights are determined by the Markov partition and the lift map.
Moreover, our results on TMC does not depend on the fact that it is originated from 
a circle map as it explained in Section~\ref{sec_indep}.

The article is organized as follows. In Section~\ref{sec_def} we give the definitions of 
the directional complexity and the directional entropy $H_\alpha$ for a map of the circle. 
In Section~\ref{sec_rot}
 we show that $H_\alpha\neq 0$ only if $\alpha$ belongs to the rotation interval. In 
Section~\ref{sec_pie} we define piecewise affine Markov maps and show how to calculate the
$(\epsilon,n)$-complexity in terms of symbolic dynamics. Section~\ref{sec_com} is devoted to the
description of the combinatorial machinery. In Section~\ref{sec_imp} we describe a specific 
example where all can be explicitly seen. In Section~\ref{sec_mea} we construct some invariant 
probabilistic measures for which measure theoretical entropies coincide with the directional 
entropies. By using this we show that the topological entropy coincides with a directional 
entropy for some specific direction. We present a formula for this direction.
Section~\ref{sec_indep} is devoted to the definition of directional entropy for 
topological Markov chains.
Section~\ref{sec_con} contains some concluding remarks.    
\section{Definitions}\label{sec_def}
Let $f: S^1 \rightarrow S^1$, $S^1=\Real/\Z$ be a continuous mapping of
degree one, i.e. there is a lift mapping $F: \Real^1 \rightarrow \Real^1$ of the form
\begin{equation}\label{Eq1}
F(x) = x + w+ h(x),
\end{equation}
where $h$ is $1$-periodic function such that $\int_0^1 h(x)dx=0$. Thus, $f(x)=x+w+h(x)
\mod 1$.

Let ${\mathbf e}=(e_x,e_y)$ be the unit vector in direction $\alpha$, that is
${\mathbf e}=\sqrt{\frac{1}{1+\alpha^2}}(\alpha,1)$.
Given $l_1 < l_2$, let 
\begin{displaymath}
W = W(l_1,l_2,\alpha) = \{(x+t e_x, t e_y)\;\;|\;\; 0 \leq t,\; l_1\leq x\leq l_2\}
\end{displaymath}
be the ``window'' in $\Real \times \Real^+$. 
\begin{definition}
\cite{AZ} \begin{itemize}
\item[1)] Two points $x,y\in \Real$ are $(\epsilon,W,T)$-separated 
if $(F^nx,n), (F^ny,n)\in W$ for each $n \leq T$, and there exists $0\leq n\leq T$
such that $|F^nx - F^ny|\geq \epsilon$.
\item[2)] A set $X\subset \Real$ is $(\epsilon,W,T)$-separated if any pair
$x,y$ in $X$, $x\neq y$, is $(\epsilon,W,T)$-separated.
\item[3)] The number
\begin{displaymath}
C_\epsilon(W,T)= \max\{card\; X\;\;|\;\; X\hspace{0.2cm} \mbox{is}\hspace{0.2cm}
(\epsilon,W,T)-\mbox{separated}\},
\end{displaymath}
is called the directional $(\epsilon,W,T)$-complexity (in the direction $\mathbf e$). 
Here, $card\; X$ is
the cardinality (the number of points) of $X$.
\item[4)] The number
\begin{displaymath}
\displaystyle\lim_{\epsilon \rightarrow 0}\overline{\displaystyle\lim_{T\rightarrow \infty}}
\frac{\ln C_\epsilon(W,T)}{T} = \cH_\alpha(l_1,l_2),
\end{displaymath}
is called the directional entropy in the direction ${\mathbf e}$ with respect 
to the interval $[l_1,l_2]$.
The limit 
$$
\cH_\alpha=\mathop{\lim_{l_1\to -\infty}}\limits_{l_2\to\infty} \cH_\alpha(l_1,l_2)
$$
is called the directional entropy in the direction $\mathbf e$.
\item[5)] Given a window $W$, an $(\epsilon,W,T)$-separated set $X$ is
optimal if $card\; X = C_\alpha(W,T)$.
\end{itemize}
\end{definition}
\begin{remark}
Roughly speaking, $C_\epsilon$ and $\cH_\alpha$ are quantities reflecting the number 
of orbits
``moving'' with the velocity $\alpha$ along the circle. Indeed, to be in the window $W$, the point
$(F^nx,n)$ must satisfy the inequality
\begin{equation}\label{Eq2}
l_1 + n\alpha \leq F^nx \leq l_2 + n \alpha,
\end{equation}
thus the ``velocity'' $\displaystyle{\frac{F^nx}{n}}$ is approximately $\alpha$
if $n >> 1$.
\end{remark}
\section{Rotation intervals and directional entropy}\label{sec_rot}

The ratio $\displaystyle{\frac{F^nx}{n}}$ is not only the velocity but also is
related to the rotation number of the orbit going through the point $x$.
\begin{definition}
\cite{NPT},\cite{I}. The set
\begin{displaymath}
\displaystyle\bigcup_{x\in[0,1]}\overline{lt}_{n\rightarrow \infty}\frac{F^nx}{n}= I,
\end{displaymath}
i. e., the set of all points of accumulation for all initial points $x\in [0,1]$
(the upper topological limit), is called the rotation interval of f.
\end{definition}

It is known (\cite{I},\cite{NPT},\cite{BMPT}) that the rotation interval is a closed interval and
for every $\mu\in I$ there is $x\in [0,1]$ such that $\displaystyle\lim_{n\rightarrow \infty}\frac{F^nx}{n}=\mu$.\\

\begin{lemma}\label{Th1}
The entropy $\cH_\alpha=0$ if $\alpha \notin I$.
\end{lemma}
\begin{proof}
Denote by $a$ ($b$) the left (right) endpoiont of the segment $I$. It is known (see \cite{ALM})
that there are functions $F_1,F_2:\R\to\R$ such that:
\begin{enumerate}
\item[i)] $F_{i}$ are weakly     monotone, i.e. the inequality $x<y$ implies 
$F_{i}(x)\leq F_{i}(y)$, $i=1,2$;
\item[ii)] there exist limits 
$$
\lim_{n\to\infty}\frac{F_1^n(x)}{n}=a,\;\;\lim_{n\to\infty}\frac{F_2^n(x)}{n}=b
$$
for any $x\in\R$;
\item[iii)] for any $x\in\R$ one has $F_1(x)\leq F(x)\leq F_2(x)$.
\end{enumerate}
The properties i) and ii) imply that 
\begin{equation}\label{Eq_c}
F_1^n(x)\leq F^n(x)\leq F_2(x)
\end{equation}
for every $x\in\R$ and $n\in\N$.

Assume now that $H_\alpha>0$ and $\alpha>b$. It means that there exists $\epsilon>0$ and 
$l_1<l_2$ such that $H_\alpha(l_1,l_2)>b+\alpha$. Therefore there exists $x\in\R$ such that 
the inequalities (\ref{Eq2}) hold for each $n\in\N$. The inequalities (\ref{Eq2}) and
(\ref{Eq_c}) imply that 
$$
F_2^n(x)\geq l_1+n\alpha\geq l_1+n(b+\epsilon)
$$
or
$$
\frac{F_2^n(x)}{n}\geq \frac{l_1}{n}+(b+\epsilon).
$$ 
Taking the limit as $n\to \infty$ we obtain a contradiction. In the same way we prove that 
$H_\alpha$ cannot be positive if $\alpha<a$.
\end{proof}

\section{Piecewise affine Markov maps}\label{sec_pie}

In this section we consider arbitrary piecewise affine Markov maps on the circle. For that, we
represent $S^1$ as $\R/\Z$ or as the interval
$[0,1]$ with the identified 
endpoints. Let $\mathcal{D}=\{
d_0=0<d_1<\dots<d_p=1\}$, $i=0,\dots,p-1$, be an
ordered collection of points on $S^1$. We introduce the following
class of maps $f:S^1\rightarrow S^1$:
\begin{itemize}
\item[(i)] $f$ is a continuous map of degree $1$,
\item[(ii)] $f(\mathcal{D})\subset \mathcal{D}$,
\item[(iii)] $f$ is an affine map on each interval $[d_i,d_{i+1}]$:
$f(x)=a_i x + b_i$, $i=0,\dots,p-1$, $a_i \neq 0$,so, in
particular $f$ is one-to-one on $[d_j,d_{j+1}]$
\item[(iv)] $|f'(x)| > 1$, $x \notin \mathcal{D}$, or $|a_i| > 1$,
$i=0,\dots,p-1$.
\end{itemize}

Remark that the condition $(ii)$ says that the points $\mathcal{D}$
determine a Markov partition for $f$ on $S^1$, and the condition (iv)
claims that $f$ is expanding on each element of this partition. Let
us emphasize that this class of maps is interesting and large
enough: first of all, Markov maps are dense in the space of
expanding maps endowed with the topology of uniform convergence, and
second, any Markov expanding map is semi-conjugated to a piecewise
affine Markov map (is conjugated in the transitive case), see, for
instance, \cite{ALM}.\\

Given $f$ of this class, let us choose the lifting map $F:\Real
\rightarrow \Real$ such that $F(0)\in [0,1]$, $F(1)\in [1,2]$. Since
$f$ is of degree 1, such a lift always exists.\\

Let $\xi_i=[d_{i},d_{i+1})$ be the $i$-th element of the Markov
partition $\xi$, $i=0,\dots,p-1$. Without loss of generality one may
assume that $diam\; F(\xi_i)< 1$, $i=0,\dots,p-1$. If it is not
so, one may consider the dynamical refinement $\xi^{(n)}= \xi \cap
f^{-1} \xi \cap \dots f^{-n+1}\xi$. Because of the condition $(iv)$,
the diameter of an element $\xi^{(n)}$ goes to $0$ as $n \rightarrow
\infty$, so one may find out $n_0$, such that $diam\; F(\xi_j^{n_0})
< 1$ for every element $\xi_j^{(n_0)}\in \xi^{n_0}$ and treat
$\xi^{(n)}$ as the original partition $\xi$. Because of that, one
may see that, first, if $f(int\; \xi_i)\cap int\; \xi_j \neq 0$ then
$f(int\; \xi_i)\supset int\; \xi_j$  ($int\;\xi_i=(d_i,d_{i+1})$,
the open interval), and, second, for $x\in \xi_j$  the set $f^{-1}x\cap \xi_i$ 
consists of exactly one point if 
$f(int\; \xi_i)\cap\xi_j\neq\emptyset$, and $f^{-1}x\cap \xi_i=\emptyset$ if 
$f(int\; \xi_i)\cap\xi_j=\emptyset$ .\\

As usual, we identify the elements $\xi_i$ with the symbols $i$,
consider the $p\times p$-matrix $A=(a_{ij})$, $a_{ij}=1$ iff
$f(int\;\xi_i)\cap int\;\xi_j \neq \emptyset$, and introduce the
one-sided topological Markov chain $(\Omega_{A},\sigma)$ where
$\Omega_{A}=
\{\underline{\mathbf{\omega}}=(\omega_0\;\omega_1\;\dots\;\omega_k\;\dots)\;\;|\;\;\omega_k
\in \{0,1,\dots,p-1\}$, $\omega_k$ can follow $\omega_{k-1}$ iff
$a_{\omega_{k-1}\omega_k}= 1$, $k=1,\dots\}$. We endow $\Omega_{A}$ with the
distance
\begin{displaymath}
d(\underline{\mathbf{\omega}},\underline{\mathbf{\omega}}')=
\displaystyle\sum_{k=0}^{\infty}\frac{|\omega_k-\omega_{k}'|}{p^k},
\end{displaymath}
so, the shift map $\sigma:\Omega_{A}\rightarrow \Omega_{A}$,
$(\sigma\underline{\mathbf{\omega}})_k=\omega_{k+1}$, $k\in \Integer_{+}$, will
be continuous. The coding map $\chi:\Omega_{A}\rightarrow S'$ is
well-defined in such a way that, for $\underline{\mathbf{\omega}}=
(\omega_0\;\omega_1\;\dots)\in \Omega_{A}$
\begin{displaymath}
\chi(\underline{\mathbf{\omega}})=
\displaystyle\bigcap_{n=1}^{\infty}\Delta_{\omega_0\dots \omega_{n-1}}
\end{displaymath}
where $\Delta_{\omega_0\dots \omega_{n-1}}=\xi_{\omega_0}\cap
f^{-1}\xi_{\omega_1}\cap\dots\cap f^{-n+1}\xi_{\omega_{n-1}}$. Since, for
$\underline{\mathbf{\omega}}\in\Omega_A$, 
$$
diam\;\Delta_{\omega_0\dots\omega_{n-1}}=
\prod_{k=0}^{n-1} |a_{\omega_k}^{-1}|\rightarrow 0
\;\;\mbox{as}\;\; n\rightarrow \infty,
$$ 
then $\chi(\underline{\mathbf{\omega}})$ consists of
the only one point.\\

\subsection{Estimates from above} \label{subsec_estimate}

We introduce an oriented graph $\Gamma_A$ having $p$ vertices such
that there exists an edge starting at the vertex $i$ and 
ending at
$j$ iff $a_{ij}=1$. By $L^*_\Gamma$ we denote all $\Gamma$-admissible finite words 
(paths: $(\omega_0,\omega_1,\dots,\omega_{n-1})\in L^*_\Gamma$ iff $(\omega_{j-1},\omega_j)$ is a $\Gamma$-edge
for all $j=1,\dots,n-1$). As the graph $\Gamma$ is normally fixed we sometimes omit 
the subscript $\Gamma$.
We relate a weight $k_{ij}\in \Integer$ to every edge $(i\;j)$
of the graph $\Gamma_A$ as follows: $k_{ij}=s$ iff $F(\xi_i)\supset \xi_j
+ s$ where $\xi_j + s = \{x+s\;\;|\;\;x\in\xi_j\}$. Since $F$ is continuous,
the collection $\{k_{ij}\; |\; a_{ij}=1\}=\{s_0,s_0+1,\dots,s_0+\rho\}$, 
$s_0\leq 0$, $-s_0,\rho\in\N$. 
Now we want to estimate $C_\epsilon(W,T)$ through the cardinality of different sets of words 
generated
by $\Gamma_A$. Let us start with some notation and definitions.
For a finite word  $w=w_0\dots w_{n-1}\in L^*_\Gamma$ we denote:
\begin{itemize}
\item $|w|=n$, the length  of the sequence.
\item $w[i:j]=w_iw_{i+1}\dots w_j$; $w[:j]=w_0\dots w_j$.
\item $v(w)=\sum\limits_{i=1}^{n-1}k_{(w[i-1,i])}$, the weight of $w$.
\item $L^n =\{w\in L^*_\Gamma\;|\;|w|=n\}$, the collection of all admissible words of length $n$. 
\item $L^n_{m}=\{w\in L^n\;|\;v(w)=m\}$, the collections of admissible $n$-words of the weight $m$.
\item For any $w\in L^n$ let $[w]\subseteq \Omega_A$ be the corresponding cylinder, i.e.
$[w]=\{\underline{\omega}\in\Omega_A\;|\;\omega[:n-1]=w\}$.
\end{itemize}
\begin{lemma}\label{Lem1}
Given $w \in L^n$, for
any $x\in \chi([w])=\Delta_{w_0 \dots w_{n-1}}$ one has
\begin{equation}\label{(b)}
m\leq F^{n-1}x \leq m+1,
\end{equation}
where $m=v(w)$
\end{lemma}

\begin{proof}
In fact, the statement directly follows from the definition of
$k_{ij}$. Indeed, if $0\leq x\leq 1$ then $Fx \in [k_{w_0 w_1}, k_{w_0 w_1}+1]$ and so on.
\end{proof}
\begin{proposition}\label{Prop5}
If, for $x\in [0,1]$, the inequality \eqref{(b)} is satisfied then $x\in
\chi([w])$, $w \in L^{n}_{m-1} \cup L^{n}_m \cup
L^{n}_{m+1}$.
\end{proposition}

\begin{proof}
Since the images of the cylinders $\{\chi([w])\;|\;w\in L^n\}$ form a partition of the interval
$[0,1]$ then $x\in \chi([w])$, $w = w_0 \dots
w_{n-1}$. Let $q=\displaystyle\sum_{j=0}^{n-1}k_{w_j w_{j+1}}$. If
$q > m+1$ $(q < m-1)$ then, because of Lemma~\ref{Lem1}, $F^{n-1}x\geq
q > m+1$ $(F^{n-1}x \leq q+1 < m)$, the contradiction with \eqref{(b)}.
\end{proof}

For $\alpha\in\R_+$, $r,n\in\N$, let 
$B_{n,\alpha,r}=\{w\in L^{n}\;|\; \forall j=1,\dots,n-1\; \alpha j-r\leq v(w[:j])\leq \alpha j +r\}$.
The following proposition is an easy implication of the definition of $B_{n,\alpha,r}$.
\begin{proposition}\label{Prop_lev1}
Let $|w|=n$. Then $w\in B_{n,\alpha,r}$ if and only if 
for any $j=1,\dots,n-1$ one has 
$w[:j]\in \bigcup\limits_{m=\fl{\alpha j}-r}^{\fl{\alpha j+r}} L^{j+1}_m$.  
\end{proposition}
We want to estimate $C_\epsilon(W(\alpha,[-r,r]),n)$ using the cardinalities of the
sets $B_{n,\alpha,r+1}$.

\begin{lemma}\label{Th5}
The following estimate holds
\begin{equation}\label{(b+1)}
C_\epsilon(W(\alpha,[-r,r]),n)\leq
\Big[\frac{1}{\epsilon}\Big] |B_{n,\alpha,r+1}|
\end{equation}
\end{lemma}

\begin{proof}
Let $\mathcal{P}$ be a an $(\epsilon,W,n)$-separated optimal set.
By definition, if $x\in \mathcal{P}$ then $(t-1)\alpha-r \leq F^{t-1}x \leq
(t-1)\alpha+r$ for $t=1,\dots,n$. Now, $x\in \chi([w])$ where
$w=w_0 \dots w_{t-1}$. Because of Proposition~\ref{Prop5},
$$
w \in \bigcup\limits_{m=\fl{(t-1)\alpha}-r-1}^{\fl{(t-1)\alpha}+r+1}L^t_m.
$$
So, by Proposition~\ref{Prop_lev1}, $x\in\Delta_w$ with $w\in B_{n,\alpha,r+1}$. 
Since $F^{n-1}$ is one-to-one on $\Delta_{w_0\dots w_{n-1}}$ and
$|F'(y)| > 1$ then $|F^{t-1}x - F^{t-1}y|\geq
\epsilon$ for some $t<n$ and $x,y\in \Delta_{w_0\dots w_{n-1}}$, implies $|F^{n-1}x-F^{n-1}y|
\geq \epsilon$.

Since $F^{n-1}\Delta_{w_0 \dots w_{n-1}}$ is an interval of length less than 1, 
the number of points of $\mathcal{P}$ inside
$\Delta_{w_0 \dots w_{n-1}}$ does not exceed
$\Big[\frac{1}{\epsilon}\Big]$. Thus
\begin{equation*}
|\mathcal{P}|=C_\epsilon(W)\leq
\Big[\frac{1}{\epsilon}\Big]|B_{n,\alpha,r+1}|.
\end{equation*}
\end{proof}
\subsection{An estimate from below}

Let $m\in\N$. The set $\{\Delta_w\;|\;w\in L^m\}$ is a partition of $[0,1]$ by intervals.
Let $\epsilon_m$ be the minimal length of the intervals $\Delta_w$, $w\in L^m$. 
\begin{lemma}\label{th_est2}
$$
C_{\epsilon_m}(W(\alpha,[-r,r]),km))\geq 3^{-k}|B_{km,\alpha,r}|
$$
\end{lemma} 
\begin{proof}
Let $S\subset L^m$ satisfy the following property:
$\forall w,v\in S, w\neq v\;\exists 0\leq j<k\;:\;$ 
$$
\Delta_{v[jm:(j+1)m]}\mbox{ and }
\Delta_{w[jm:(j+1)m]}\mbox{ are different and not successive}.
$$
Fix a maximal $S$ satisfying this property.
One can check that $x\in \Delta_w$ and $y\in \Delta_v$ are $(\epsilon,W,km)$-separated for $w,v\in S$ and $w\neq v$.
So,  $C_{\epsilon_m}(W(\alpha,[-r,r]),km))\geq |S|$. We only need to estimate $|S|$.
For $w\in B_{km,\alpha,r}$ let 
$$
U(w)=
\{v\in B_{km,\alpha,r}\;|\; \forall  0\leq j<k\;\mbox{the intervals }\Delta_{v[jm:(j+1)m]}
\mbox{ and }
\Delta_{w[jm:(j+1)m]} 
$$
$$
\mbox{ are equal or successive}\}  
$$
Observe that $|U(w)|\leq 3^k$ and $ B_{km,\alpha,r}=\bigcup\limits_{w\in S} U(w)$ due to 
the maximality of $S$.
The estimate follows.
\end{proof}
\begin{theorem}
Let
$$
e_{\alpha,r}=\ln\lim_{n\to\infty}\sqrt[n]{|B_{n,\alpha,r}|}.
$$
Then the entropy
\begin{equation}\label{eq_entrop}
\cH_\alpha=\lim_{r\to\infty} e_{\alpha,r}.
\end{equation}
\end{theorem}
\begin{proof}
Let 
$$
\overline{\displaystyle\lim_{n\rightarrow \infty}}
\frac{\ln C_\epsilon(W(\alpha,[-r,r]),n)}{n} = \cH_\alpha(\epsilon,r).
$$
Lemma~\ref{Th5} and Lemma~\ref{th_est2} together say that
$$
3^{-k}|B_{km,\alpha,r}|\leq C_\epsilon(W(\alpha,[-r,r]),km)\leq \Big[\frac{1}{\epsilon}\Big] |B_{n,\alpha),r+1}|,
$$
for $\epsilon\leq \epsilon_m$. Taking $\ln(\sqrt[km]{\cdot})$ from all parts of the above inequality and directing $k\to\infty$
one gets
$$
{-\frac{1}{m}}\ln(3)+e_{\alpha,r}\leq \cH_\alpha(\epsilon,r)\leq e_{\alpha,r+1}.
$$
The smaller $\epsilon$ is the larger $m$ can be taken ($\epsilon\leq\epsilon_m\to 0$ when $m\to \infty$). So,
$$
e_{\alpha,r}\leq\lim_{\epsilon\to 0} \cH_\alpha(\epsilon,r)\leq e_{\alpha,r+1}.
$$ 
Finally we obtain the formula (\ref{eq_entrop})
\end{proof}
\begin{remark}
We believe that formula (\ref{eq_entrop}) can be obtained by using the technique developed by M. Misiurewicz (see, for instance \cite{ALM}).
But, since we deal generally with non-invariant sets, this technique should be adjusted to the  ``non-invariant situation''. So, we decided to make
a direct proof here.
\end{remark}
\section{Combinatorial part}\label{sec_com}
Let 
$$
e_\alpha=\log\lim_{n\to\infty}\sqrt[n]{|L^n_{\fl{\alpha n}}|}.
$$
The aim of this subsection is to show that (under some conditions) 
$$
e_\alpha=\lim_{r\to\infty} e_{\alpha,r}=\cH_\alpha\;\;\;
$$
and to explain how to calculate $e_\alpha$.

Let $D\subset L^*$ be finite subset. Let the matrix $M(D)\in\Mat_{p\times p}(\N)$ be such that
$M(D)_{ij}$ is the number of words in $D$ starting from $i$ and ending by $j$.  
Given $X,Y\subset L^*$ and $B\in\Mat_{p\times p}\{0,1\}$ let $X\c{B} Y =
\{uv\;\;|\;\;u=u_1\dots u_n\in X,\; v_1\dots v_m\in Y\; B(u_n,v_1)=1\}$.
The following proposition is a direct corollary of the above definitions.
\begin{proposition}\label{prop_M-matrix}
$$
M(X\c{B} Y)=M(X)BM(Y)
$$
\end{proposition}
Recall that $L^n$ is the set of admissible words related to matrix $A$. 
It is known that $M(L^n)=A^{n-1}$, see, for instance, \cite{AH}.
Let us represent the matrix $A$ in the form
$$
A=\sum_{s\in S} A_s
$$ 
according to weight of the edges of $\Gamma$.
Precisely, $A_s\in\Mat(\{0,1\})$, $A_s(i,j)=1$ if and only if $k_{ij}=s$. 
Here the set $S$ is the set of all possible weights. 
\begin{proposition}\label{prop_M(L)}
\begin{enumerate}
\item[i)] $M(L^1_0)=E$  and $M(L^1_m)=0$ if $m\neq 0$.
\item[ii)] For $n\in\Z_+$ the following equality holds
$$
M(L^{n+1}_m)=\sum_{s\in S}M(L^{n}_{m-s})A_s,
$$
\end{enumerate} 
\end{proposition}
\begin{proof}
By definition  $L^1=\{0,\dots,p-1\}$. Any word of length $n+1$ has a form $wj$, where $w$ is 
a word of length $n$ and $j\in\{1,\dots,p\}$ and $v(wj)=v(w)+v(w_{n-1}j)$. So, one has
$$
L^{n+1}_m=\bigcup_{s\in S} L^{n}_{m-s}\c{A_s}\{0,\dots,p-1\}.
$$
So, Proposition~\ref{prop_M-matrix} implies the statement.
\end{proof}

The following proposition is a consequence of definition of $B_{n,\alpha,r}$ and
$L^n_m$. 
\begin{proposition}\label{Prop_est1}
\begin{equation}\label{eq_set1}
B_{n,\alpha,r}\subset \bigcup_{m=\fl{(n-1)\alpha}-r}^{\fl{(n-1)\alpha}+r} L^n_m,
\end{equation}
\end{proposition}
For $B_j\in\Mat_{p\times p}\{0,1\}$ we use below the notation 
$(X_1\c{B_1}\cup X_2\c{B_2})Y=X_1\c{B_1}Y\cup X_2\c{B_2}Y$.
\begin{proposition}\label{Prop_est2}
Fix $t\in \N$ and $\alpha\in\R$. Let $r\in \N$ be large enough 
($r>(t-1)\cdot(\max\{|s-\alpha|\;|\;s\in S\})$), 
$m_j=\lfloor jt\alpha \rfloor - \lfloor (j-1)t\alpha \rfloor$. 
Then for any $c\in \N$ one has:
\begin{equation}\label{eq_set2}
(\bigcup_{s\in S} L^t_{m_1-s}\c{A_{s}})(\bigcup_{s\in S} L^t_{m_2-s}\c{A_{s}})\dots 
(L^t_{m_c})
\subset B_{ct,\alpha,r}.
\end{equation}
Moreover, $m_j=\fl{t\alpha}$ or $m_j=\fl{t\alpha}+1$ for $j=0,1\dots c$.
\end{proposition}
\begin{proof}
The words of the set $B_{ct,\alpha,r}$ are the words such that the  weights  
of their initial subwords are in the [-r,r]-strip with slope $\alpha$. 
In the words from l.h.s. of the equation (\ref{eq_set2}) we fix the weights of the initial 
subwords with the length being 
multiple of $t$. Because $r$ is large enough the weights have no chance 
to leave the [-r,r]-strip. Now we make the corresponding calculations.  
Let $w$ be in l.h.s. of the inclusion. It means that $v(w[:t])=m_1$,
$v(w[t:2t])=m_2$. Generally, $v(w[(j-1)t:jt])=m_j$ for $j=1,\dots,c-1$, and
$v(w[(c-1)t:ct-1])=m_c$. So, $v(w[:jt])=m_1+m_2+\dots+m_j=\fl{jt\alpha}$.
Now, $|v(w[:jt+k])-\alpha(jt+k)|< 1+max_s|s-\alpha|k$. Here $k\leq t-1$, so
$w\in B_{ct,\alpha,r}$.

\end{proof}

For two matrices $M,N$ of  the same size over $\Z$ we write $M\leq N$ if $M_{ij}\leq N_{ij}$ 
for all
admissible indexes. The equations (\ref{eq_set1}) (\ref{eq_set2}) imply the following 
inequalities for $M$-matrices:
$$
\left( \sum_{s\in S}M(L^t_{m_1-s})A_s\right)
\left(\sum_{s\in S}M(L^t_{m_2-s})A_s\right)
\dots M(L^t_{m_c})
\leq 
$$
\begin{equation}\label{Eq_mat}
M(B_{ct,\alpha,r})\leq \sum_{m=\fl{(ct-1)\alpha}-r}^{\fl{(ct-1)\alpha}+r} M(L^{ct}_{m})
\end{equation}
Applying Proposition~\ref{prop_M(L)} to this inequality we obtain
\begin{proposition}\label{prop_matr_ineq}
$$
M(L^{t+1}_{m_1})M(L^{t+1}_{m_2})\dots M(L^{t+1}_{m_{c-1}})M(L^t_{m_c})
\leq M(B_{ct,\alpha,r})\leq \sum_{m=\fl{(ct-1)\alpha}-r}^{\fl{(ct-1)\alpha}+r} M(L^{ct}_{m}),
$$
where $m_j=\fl{jt\alpha}-\fl{(j-1)t\alpha}$. Moreover, $m_j=\fl{t\alpha}$ or
$m_j=\fl{t\alpha}+1$.  
\end{proposition} 

For a positive sequence $a_n$ we call $\lim\limits_{n\to\infty}\sqrt[n]{a_n}$ the exponent of $a_n$ 
(if exists).
The relation between exponents of $D_n$ and $M(D_n)$ is clear:
$$
\lim\sqrt[n]{|D_n|}=\max_{ij}\{\lim\sqrt[n]{m_{ij}(n)}\}, 
$$
where $m_{ij}$ are matrix entries
of $M(D_n)$. Using this fact and estimates of Proposition~\ref{prop_matr_ineq}
one gets $e_{\alpha,r}\leq\lim\limits_{\epsilon\to 0}\sup\{e_\beta\;|\;\beta\in 
[\alpha-\epsilon,\alpha+\epsilon]\}$. So, the following lemma holds.
\begin{lemma}\label{lemma_above}
If $e_\alpha$ depends continuously on $\alpha$, then $e_{\alpha, r}\leq e_{\alpha}$.
\end{lemma}
The estimates from below may be more tricky to obtain. We overcome this difficulty 
by imposing a rather general sufficient condition.
\begin{lemma}\label{lemma_below}
Let $M(L^n_{\fl{\alpha n}})$ 
have a diagonal entry
with exponent $e_\alpha$ then $\lim\limits_{r\to\infty} e_{\alpha,r}\geq e_{\alpha}$.
\end{lemma}
\begin{proof}
Let $M(L^n_{\fl{\alpha n}})_{jj}$ be a diagonal entry with exponent $e_\alpha$.  
Let 
$$
d(t)=\min\{(M^{t+a}_{\fl{\alpha t}+b})_{jj}\;|\;a,b=0,1\}.
$$ 
Then Proposition~\ref{prop_matr_ineq}
implies the inequality 
$$
d(t)^c\leq M_{jj}(B_{ct,\alpha,r}).
$$ 
Applying $\sqrt[ct]\cdot$ and allowing $c\to\infty$ one gets
$$
\sqrt[t]{d(t)}\leq \lim_{r\to\infty}e_{\alpha,r},
$$
But $\sqrt[t]{d(t)}\to e_\alpha$ by our assumptions.
\end{proof}
In the next subsection we explain how to calculate $M(L^n_{\fl{\alpha n}})$.

\subsection{Generating function.}
Let $S=\{s_0,s_0+1,...,s_0+\rho\}$. We define
the matrix generating function for $M(L^n_m)$ as
$$
G(x,y)=\sum\limits_{n=1}^\infty\sum\limits_{m=(n-1)s_0}^{(n-1)(s_0+\rho)}
M(L^n_m)x^{n-1}y^{m-(n-1)s_0}.
$$
 We chose this type of generating
function to avoid negative powers and to keep track of the number of
total transitions.

\begin{lemma}\label{lm_gen2}
$
G(x,y)=\big(E-x(A_{s_0}+yA_{s_0+1}+...+y^iA_{s_0+i}+...+y^\rho A_{s_0+\rho})\big)^{-1}
$
\end{lemma}
\begin{proof}
Taking into account the formula $(E-X)^{-1}=E+X+X^2\dots$ it suffices to show that
$$
(A_{s_0}+yA_{s_0+1}+...+y^iA_{s_0+i}+...+y^\rho A_{s_0+\rho})^n=\sum\limits_{m=ns_0}^{n(s_0+\rho)}
M(L^{n+1}_m)y^{m-ns_0}.
$$
We prove it by induction on $n$. For $n=0$ the equality holds by the statement {\bf i)} of
Proposition~\ref{prop_M(L)}. Supposing the equality for 
$n-1$ we obtain
$$
(A_{s_0}+yA_{s_0+1}+...+y^iA_{s_0+i}+...+y^\rho A_{s_0+\rho})^n=
$$
$$
\left(\sum\limits_{m=(n-1)s_0}^{(n-1)(s_0+\rho)} M(L^{n}_m)y^{m-(n-1)s_0}\right)
(A_{s_0}+yA_{s_0+1}+...+y^iA_{s_0+i}+...+y^\rho A_{s_0+\rho})=
$$ 
$$
\left(\sum\limits_{m=ns_0}^{ns_0+(n-1)\rho} M(L^{n}_{m-s_0})y^{m-ns_0}\right)
(A_{s_0}+yA_{s_0+1}+...+y^iA_{s_0+i}+...+y^\rho A_{s_0+\rho})=
$$
$$
\sum_{m=ns_0}^{ns_0+n\rho}\left(\sum_{j=0}^\rho M(L^n_{m-s_0-j})A_{s_0+j}\right)y^{m-ns_0}.
$$
In the last equality we use the simple fact that $L^n_m=\emptyset$ for $m< (n-1)s_0$ and
$m>(n-1)(s_0+\rho)$. 
The induction step follows because of Proposition~\ref{prop_M(L)}. 
\end{proof}

Let $H(x,y)=\det(E-x\sum\limits_{j=0}^\rho y^jA_{s_0+j})$. It follows from the formula 
of an inverse 
matrix that $HG$ is a polynomial matrix.
In order to calculate the asymptotics we need to study the zeros of $H$,
particularly, we need the so called minimal solutions, see \cite{P1,P2,Pemantle_preprint}.
\begin{definition}
Let $f(x,y)$ be a $\C$-polynomial. Consider the equation
\begin{equation}\label{Eq_min}
f(x,y)=0
\end{equation}
A solution $(x_0,y_0)\in\C^2$ of (\ref{Eq_min}) is said to be minimal if 
equation (\ref{Eq_min}) has no solution $(x,y)$ satisfying $|x|<|x_0|$ and $|y|<|y_0|$.
A solution $(x_0,y_0)\in\C^2$ of the equation (\ref{Eq_min}) is said to be strictly minimal if
the inequalities $|x|\leq |x_0|$ and $|y|\leq |y_0|$ for any solution $(x,y)$ imply $x=x_0$, $y=y_0$. 
\end{definition}
The following proposition describes the minimal solutions for
\begin{equation}\label{Eq_min2}
H(x,y)=0
\end{equation}
\begin{proposition}\label{Prop_min}
 Let $A$ be a primitive matrix. Let $(x_0,y_0)\in\C^2$, $y_0\neq 0$ be 
a minimal solution of the  equation (\ref{Eq_min2}).
Then the maximal (by the absolute value) eigenvalue of the
matrix $A(x_0,y_0)=x_0\sum\limits_jy_0^jA_{s_0+j}$ is $1$. 
Moreover, if rank of $(A(1,e^{i\phi}))$ $>1$ for all $\phi\in\R$ then $(x_0,y_0)\in\R^2_+$
and $(x_0,y_0)$ is strictly minimal.
\end{proposition}
\begin{proof}
Clearly, $H(x_0,y_0)=0$ iff $1$ is an eigenvalue of 
$A(x_0,y_0)$. If $\lambda$ is an eigenvalue of $A(x_0,y_0)$ with $|\lambda|>1$ then
$H(x_0/\lambda,y_0)=0$, a contradiction with the minimality of $(x_0,y_0)$. 

For a vectors $u,v\in\R^p$ we write $u\geq v$ if $u_i\geq v_i$ for all $i=1,\dots,p$.
We write $u>v$ if $u\geq v$ and $u\neq v$.
Let $(x_0,y_0)\not\in\R^2$ and $A(x_0,y_0)\xi=\xi$ for $\xi\in\C^p$. 
Define $v\in\R^p$ as $v_i=|\xi_i|$. Observe that
$A(|x_0|,|y_0|)v\geq v$. If $A(|x_0|,|y_0|)v>v$ then the maximal real eigenvalue of 
$A(|x_0|,|y_0|)$ is greater than $1$ by Proposition~\ref{prop_prim_1}(see below) and $(x_0,y_0)$ is 
not minimal, a contradiction. Assume now that $A(|x_0|,|y_0|)v=v$ and $A(x_0,y_0)=\{a_{jk}\}$. 
It follows that $\arg(a_{jk}\xi_k)=\arg(\xi_j)$, or, the same, 
$\arg(a_{jk})=\arg(\xi_j)-\arg(\xi_k)$.
In our situation it means that $A(1,e^{i\phi})_{jk}=e^{i(\phi_0+\phi_j-\phi_k)}$, where
$\phi_0=\arg(y_0)$ and $\phi_j=\arg(\xi_j)$. So, $\rank(A(1,e^{i\phi}))=1$, a contradiction. 
 \end{proof}
\begin{proposition}\label{prop_prim_1}

Let $b(A)$ be the greatest real eigenvalue of a matrix $A$. Let
$A$ be primitive and $Av>v$ for some $v>0$. Then $b(A)>1$.
\end{proposition}
\begin{proof}
There exists $n$ such that all entries of $A^n$ are positive. Observe that
if $u>v$ then $(A^nu)_i>(A^nv)_i$ for all $i=1,\dots,p$. Observe also that $A^nv>v$. Thus, 
there exists
$\beta>1$ such that $A^{2n}v>\beta A^{n}v$. Inductively, $A^{kn}v>\beta^{k-1}A^nv$. 
Recall that 
$b(A)=\lim\limits_{m\to\infty}\sqrt[m]{\|A^m\|}$. So, $b(A)\geq\sqrt[n]{\beta}>1$.
\end{proof}
\subsection{Asymptotics for 2-variable generating functions.} \label{sec_asymp2}

In this section we suppose that $A$ is primitive and the rank condition of 
Proposition~\ref{Prop_min} is satisfied. 
All entries of $G(x,y)$ have the form 
$\frac{f(x,y)}{H(x,y)}$, where $f$ is a polynomial. We are interesting in 
asymptotics of $a_{n,\fl{\alpha n}}$ where $a_{n,m}$ are the coefficients of the 
expansion 
$$
\frac{f(x,y)}{H(x,y)}=\sum a_{n,m}x^ny^m
$$ 
We estimate $a_{n,m}$ using the Wilson-Pemantle technique \cite{P1,P2}. 
The asymptotics depend on minimal points. Under the conditions of Proposition~\ref{Prop_min}
all minimal
points are strictly minimal and we may adapt Theorem~3.1 of \cite{P1} (see also
\cite{Pemantle_preprint,P2}) as follows
\begin{theorem}\label{th_comb1}
Let $(x_0,y_0)\in\R_+^2$ be the unique (in $\R_+^2$) solution of 
\begin{equation}\label{Eq_asym}
\left\{\begin{array}{l} H=0\\
                       \alpha x\partial_x H=y\partial_y H
\end{array}\right. ,
\end{equation}
such that $1$ is a maximal eigenvalue of $A(x_0,y_0)$. 
Then $(x_0,y_0)$ is a strictly minimal solution of the equation (\ref{Eq_min2}) and 
the following asymptotics takes place: 
$$
a_{n,\fl{\alpha n}}\sim \frac{f(x_0,y_0)}{\sqrt{2\pi}}x_0^{-n}y_0^{-\alpha n}
\sqrt{\frac{-x\partial_xH(x_0,y_0)}{nQ(x_0,y_0)}},
$$
where
$Q(x,y)=-xH_x(yH_y)^2-yH_y(xH_x)^2-y^2x^2[(H_y)^2H_{xx}+(H_x)^2H_{yy}
-2H_xH_yH_{xy}]$.
Particularly, it implies that 
$$
\lim_{n\to\infty}\frac{\ln(a_{n,\fl{\alpha n}})}{n}=-ln(x_0)-\alpha\ln(y_0),
$$ 
if $f(x_0,y_0)\neq 0$ and $Q(x_0,y_0)\neq 0$.
\end{theorem}
In the following, we assume, without loss of generality, that $s_0=0$. (If not,
one should make a change $\alpha\to \alpha-s_0$. 
Theorem~\ref{th_comb1} with Lemma~\ref{lemma_above} and Lemma~\ref{lemma_below} imply
\begin{theorem}\label{Th_cor_main}
Let $(x_0,y_0)$ be the unique in $\R^2_+$ solution of the system 
(\ref{Eq_asym}).
Let the polynomial matrix $HG$ have a non-zero diagonal entry evaluated at $(x_0,y_0)$ 
and $Q(x_0,y_0)\neq 0$. 
Then
$\cH_\alpha=-ln(x_0)-\alpha\ln(y_0)$. 
\end{theorem}

\section{Example} \label{sec_imp}

In this section we consider an example that, in fact, contains all main features 
of systems on 
the circle possessing a Markov partition. 

Consider the map $f$ for which

\begin{displaymath}
F(x)=  \left \{\begin{array}{rc}
\frac{1}{3}+2x, & 0 \leq x \leq \frac{1}{3},\\
\frac{4}{3}-x, & \frac{1}{3} \leq x \leq \frac{2}{3},\\
-\frac{2}{3}+2x, & \frac{2}{3} \leq x \leq 1.
\end{array}\right.
\end{displaymath}
\medskip

The map $f$ has the Markov partition $\xi$ of $3$ intervals:
$\xi_1=\big[0,\frac{1}{3}\big]$,
$\xi_2=\big[\frac{1}{3},\frac{2}{3}\big]$,
$\xi_3=\big[\frac{2}{3},1\big]$ (see Fig.~\ref{Figure3}), and the corresponding topological
Markov chain is determined by the transition matrix $A=
\left(\begin{array}{ccc}
0 & 1 & 1\\
0 & 0 & 1\\
1 & 0 & 1
\end{array}\right)$, corresponding to the graph $G$ (see Fig.~\ref{Figure4}).\\

\begin{figure}[h]
\begin{center}
     \includegraphics[height=4.5cm]{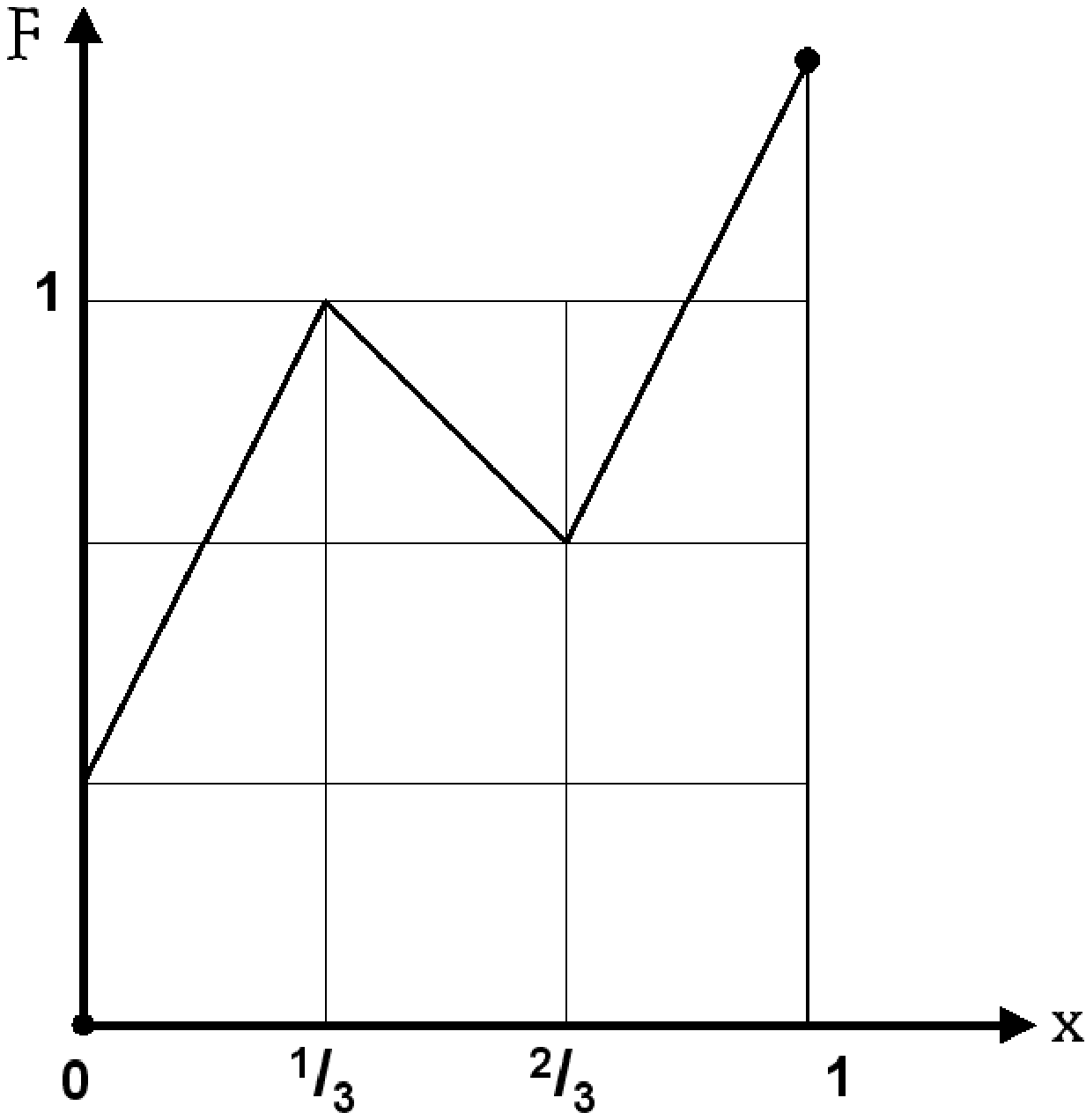}
\end{center}
\caption{\small{The graph of $F$ and the Markov partition.}}\label{Figure3}
\end{figure}

\begin{figure}[h]
\begin{center}
     \includegraphics[width=9cm]{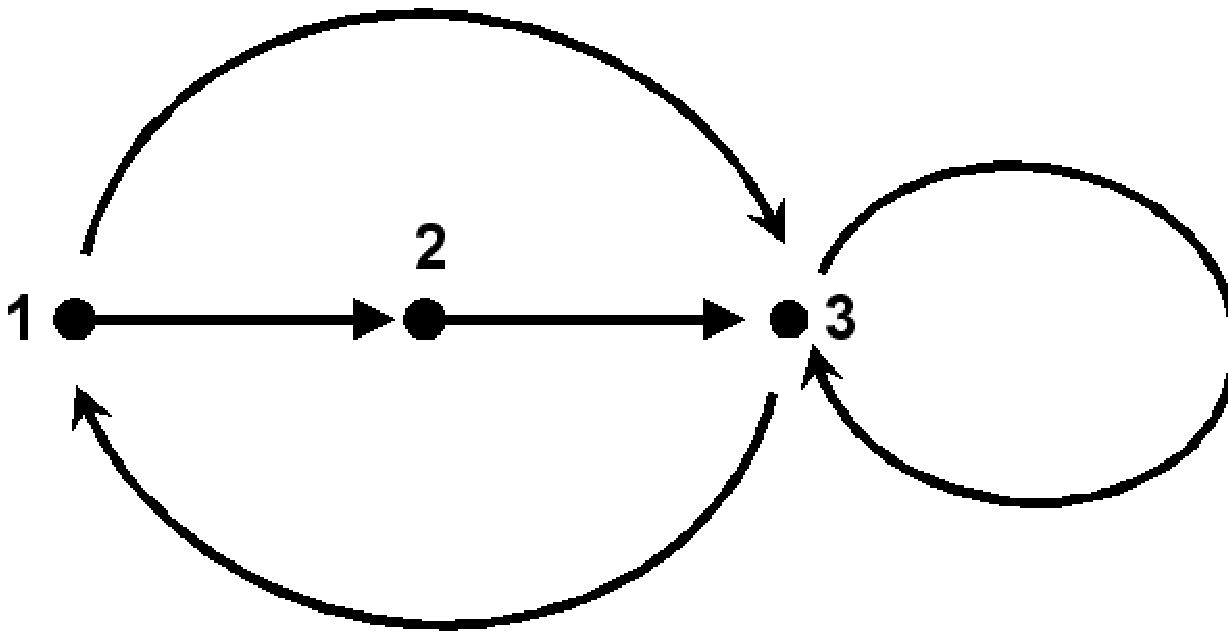}
\end{center}
\caption{\small{The oriented graph $G$ for the map $F$ and the partition $\xi$.}}
\label{Figure4}
\end{figure}
One can see that the transition $(3,1)$ corresponds to the change of the integer part
of $F$. So, we represent the transition matrix $A=A_0+A_1$ where $A_0$
corresponds to all transitions without $(31)$ and $A_1$ corresponds
to $(31)$:
$$
A_0=\left(\begin{array}{ccc}
0 & 1 & 1\\
0 & 0 & 1\\
0 & 0 & 1
\end{array}\right) \;\; A_1=\left(\begin{array}{ccc}
0 & 0 & 0\\
0 & 0 & 0\\
1 & 0 & 0
\end{array}\right)
$$
We calculate the generating function
$G(x,y)=(E-xA_0-xyA_1)^{-1}=$    
$$
\frac{1}{- x^{3} y - x^{2} y - x + 1}
\left(\begin{array}{lll}
                            - x + 1 & - x^{2} + x & x^{2} + x \\
                             x^{2} y & - x^{2} y - x + 1 & x \\
                                        x y & x^{2} y & 1
\end{array}\right)
$$

Now we can find the asymptotics using Theorem~\ref{th_comb1}. 
Let $H= - x^{3} y - x^{2} y - x + 1$
We have to find positive solutions of the system
$$
\left\{\begin{array}{l} H=0\\
           \alpha xH_x=yH_y
             \end{array}\right.
$$

Using SAGE (see \cite{SA}) we have found: 
$$
x=\frac{{\left(\alpha \pm \sqrt{5 \, \alpha^{2} - 4 \, \alpha + 1}\right)}}{{\left(2 \, \alpha -
1\right)}}.
$$
In this example $\alpha$ is a fraction of $(31)$-transition ($A_1$-transition). If $\alpha>1/2$
then $2$ consecutive $A_1$ transitions should appear. But there is no word with consecutive
$(31)$-transition. So, we have to consider the interval $0<\alpha\leq 1/2$ only. 
The positive branch for $0<\alpha<1/2$  is 
$$
x=\frac{{\left(\alpha - \sqrt{5 \, \alpha^{2} - 4 \, \alpha + 1}\right)}}{{\left(2 \, \alpha -
1\right)}}.
$$
Equation $H=0$ implies 
$$
y=\frac{1-x}{x^3+x^2}.
$$

The dependence of the entropy on $\alpha$ is given by the formula $h=-\ln(x)-\alpha\ln(y)$
shown on the figure~\ref{fig1}. One can see that our case satisfies Theorem~\ref{Th_cor_main},
so, $\cH(\alpha)=h(\alpha)$.
\begin{figure} 
\begin{center}
\includegraphics[scale=0.5]{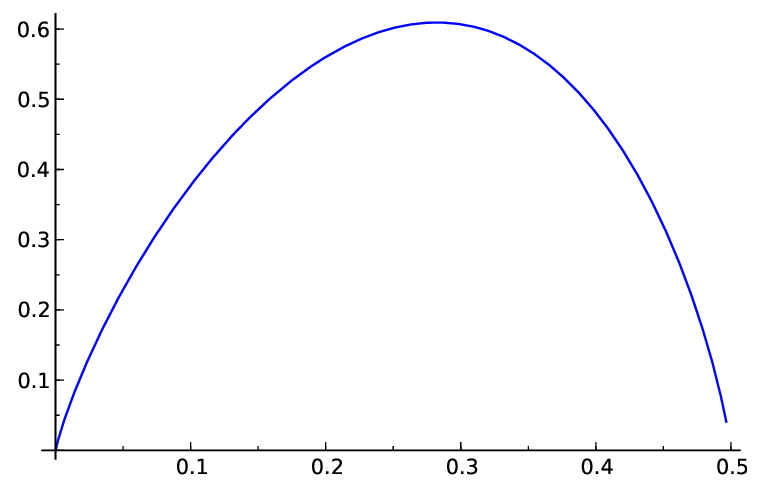}
\end{center}
\caption{The graph of $h(\alpha)$}
\label{fig1}
\end{figure}

\section{Measures and entropy} \label{sec_mea}
\subsection{Construction of the measure.}
 Recall, that under the conditions of Theorem~\ref{th_comb1} the
matrix $A(x_0,y_0)$ has $1$ as the greatest simple eigenvalue. Let $l$ be a row-vector
($r$ be a column-vector) such that $lA(x_0,y_0)=l$ ($A(x_0,y_0)r=r$). By the Perron-Frobenius
theorem $l$ and $r$ are positive. Normalize $l$ and $r$ such that $lr=1$.
Let $A(x_0,y_0)=\{a_{jk}\}$. Define (see \cite{K}) the matrix $\Pi=\Pi(x_0,y_0)$ as
$\Pi_{jk}=\frac{a_{jk}r_k}{r_j}$. Let $q_j=l_jr_j$ and $q=q_1,q_2,\dots,q_p$. 
Observe that $\Pi$ is a stochastic matrix
and $q$ is its left $1$-eigenvector. The measure $\mu_\Pi$ of the cylinder  
$[w_1,w_2,w_3,...,w_n]$ is defined as 
$$
\mu_\Pi([w_1,w_2,w_3,...,w_n])=q_{w_1}\Pi_{w_1w_2}\Pi_{w_2w_3}\dots\Pi_{w_{n-1}w_n}.
$$    
The entropy of the subshift with respect to $\mu_\Pi$ can be calculated by the formula
\begin{equation}\label{Eq_meas_entr}
h(\mu_\Pi)=-\sum_{jk}q_j\Pi_{jk}\ln(\Pi_{jk}),
\end{equation}
see \cite{K}. 
\subsection{$h(\mu_\Pi)=H_\alpha$}
We are going to show that $h(\mu_\Pi)=\ln(x_0)+\alpha\ln(y_0)$.
In our situation the equation (\ref{Eq_meas_entr}) can be rewritten as
$$
-h(\mu_\Pi)=\sum_{ik}l_ia_{ik}r_k\ln(\frac{a_{ik}r_k}{r_i})=\sum_{ik}l_ia_{ik}r_k\ln(a_{ik})+
$$
$$
\sum_{ik}l_ia_{ik}r_k\ln(r_k)-\sum_{ik}l_ia_{ik}r_k\ln(r_i).
$$
Observe that the last line of the equation is $0$. (Indeed, evaluating the first sum over $i$ and
the second one  over $k$ and taking into account 
that $l$ ($r$) is a left (right) $1$-eigenvector of $A$ we obtain that 
$\sum_kl_kr_k\ln(r_k)-\sum_il_ir_i\ln(r_i)=0$.) Let
$\cA_j=\{(i,k)\;|\;(A_{s_0+j})_{ik}=1\}$. Now we can write:
$$
-h(\mu_\Pi)=\sum_j\sum_{(i,k)\in\cA_j}l_ir_kx_0y_0^j\ln(x_0y_0^j)=
$$
$$
\ln(x_0)\sum_j\sum_{(i,k)\in\cA_j}l_ir_kx_0y_0^j+\ln(y_0)\sum_j\sum_{(i,k)\in\cA_j}l_ir_kx_0jy_0^j=
$$
$$
\ln(x_0)(lA(x_0,y_0)r)+\ln(y_0)(l \tilde A(x_0,y_0)r)=
\ln(x_0)+\ln(y_0)(l \tilde A(x_0,y_0)r),
$$
where $\tilde A(x_0,y_0)=y_0A_y(x_0,y_0)=\sum\limits_{j}jx_0y_0^jA_{s_0+j}$. So, in order to prove
the equality $\cH_\alpha=h(\mu_\Pi)$ we should show that 
$(l \tilde A(x_0,y_0)r)=\alpha$, of course, under the condition that
$lA(x_0,y_0)=l$, $A(x_0,y_0)r=r$, $lr =1$, $(x_0,y_0)$ is the solution of 
the system (\ref{Eq_asym}) satisfying the condition of Theorem~\ref{th_comb1}.

To this end we need the following result (recall that $H(x,y)=\det(E-A(x,y))$).
\begin{proposition}\label{prop_deriv_det}
Let $B\in\Mat(\C)$, $\det(B)=0$ and $0$ be a simple spectral point of $B$. 
Let $l$ be a vector-row and $r$ be a vector-column such that $lB=0$, $Br=0$, and
$lr=1$. Let $\beta=\lambda_1\lambda_2\dots\lambda_{p-1}$ be the product of all
non-zero eigenvalues of $B$ (counted with multiplicity).
Then the Frechet derivative $D\det(B)$ of $\det(B)$ (applied to an arbitrary matrix $X$) 
is iqual to
$$
D(\det(B))(X)=\beta(lXr)
$$
\end{proposition}
\begin{proof}
The multilinearity of $\det(\cdot)$ implies that 
$$
\det(B+\epsilon X)=\epsilon \sum_{ij}\tilde B_{ij}X_{ij}+O(\epsilon^2),
$$ 
where $\tilde B=\{\tilde B_{ij}\}$ is the matrix of the cofactors of $B$. 
Because of the equalities $B\tilde B^T=\tilde B^T B=\det(B)E=0$, the columns (rows) of 
$\tilde B$ are proportional to
$l$ ($r$). Thus, $\tilde B_{ij}=\gamma l_ir_j$ for some $\gamma$. Observe that 
$\gamma={\rm trace}(\tilde B)$. Let $D={\rm diag}(-1,1,-1,1\dots, (-1)^p)$. 
The matrix $D^{-1}\tilde B D$ is the matrix of the minors of $B$. By a theorem due to  
Kronecker (see \cite{Gan}) the eigenvalues of $D^{-1}\tilde B D$ (as well as of $\tilde B$) 
are products of $p-1$
eigenvalues of $B$. So, ${\rm trace}(\tilde B)=\beta$, the  unique non-zero eigenvalue 
of $\tilde B$.   
\end{proof}

Take $B=E-A(x_0,y_0)$ in Proposition~\ref{prop_deriv_det}. Then the last equation
of the system (\ref{Eq_asym}) may be rewritten as
$\alpha\beta(lA(x_0,y_0)r)=\beta (l\tilde A(x_0,y_0)r)$. But
$lA(x_0,y_0)r=1$ and we prove the following
\begin{lemma}\label{lm_measure}
$h(\mu_\Pi)=\cH_\alpha$, where $\Pi=\Pi(x_0,y_0)$ and $(x_0,y_0)$ is the minimal solution
of the system (\ref{Eq_asym}).
\end{lemma}

\begin{remark}
The direct computation shows that 
$\int_{\Omega_A}v(w[:1])d\mu_\Pi(w)=\alpha+s_0$ (the function $v(\cdot)$ is defined in 
Section\ref{subsec_estimate}). With shift invariance of $\mu_\Pi$ it probably implies  
that the support of $\mu_\Pi$ consist of initial words
with rotation number $\alpha$. 
Moreover, the measure $\mu_\Pi$ is the measure of maximal entropy among measures $\nu$ such that
$\int_{\Omega_A}v(w[:1])d\nu(w)=\alpha$. This is a manifestation of general variation principle,
see \cite{GM,Kw}. 
\end{remark}

\subsection{When $\cH_\alpha=h_{top}$.}
Lemma~\ref{lm_measure} implies that $\cH_\alpha=h_{top}$
if $\mu_\Pi$ is the measure of the maximal entropy. Observe that $A(x,1)=xA$. 
So, our construction of $\mu_\Pi$ in the case of $y_0=1$,
in fact, coincides with  the construction of the measure of maximal entropy in
\cite{K}. Substituting $y_0=1$ to the  system (\ref{Eq_asym}), we can find 
$\alpha$ and $x_0$. It is clear that, in fact, $x_0=e^{-h_{top}}$, the inverse value 
of the greatest eigenvalue of $A$ since $A(x_0,1)=x_0A$.
We can formulate the procedure of finding the angle, corresponding the topological entropy
in the form of the following 
\begin{theorem}
Let $\lambda$ be the greatest eigenvalue of $A$; $l$ ($r$) be its left (right)
$\lambda$-eigenvector. Let 
$$
\alpha=\frac{lA(1,1)r}{l\tilde A(1,1)r}+s_0.
$$
Then $H_\alpha=h_{top}$.
\end{theorem}  

In our example. $A=A_0+A_1$, where
$$
A_0=\left(\begin{array}{ccc}
0 & 1 & 1\\
0 & 0 & 1\\
0 & 0 & 1
\end{array}\right) \;\; A_1=\left(\begin{array}{ccc}
0 & 0 & 0\\
0 & 0 & 0\\
1 & 0 & 0
\end{array}\right)
$$
Denote by $\lambda$ the  maximal eigenvalue of $A$.
Let $l$ be a left $\lambda$-eigenvector of $A$ and $r$ be 
a right $\lambda$-eigenvector of $A$. Calculations show that 
$\lambda\approx 1.839$,
$$
l\approx (1, 0.5436890126920763, 1.839286755214161)
$$ 
$$
r\approx (1, 0.647798871261043, 1.191487883953119)
$$ 
(because of cancellation we do not need  
normalization here). Now let $\alpha_{max}$ be such that $H_{\alpha_{max}}=h_{top}$.
We can calculate:
$$
\alpha_{max}=\frac{lAr}{l A_1r}\approx
 	0.2821918053244515.
$$   
\section{Directional complexity and entropy for topological Markov chains}\label{sec_indep}

In Section~\ref{sec_com} we have reduced the calculation of the directional entropy for
Markov maps of the circle to the calculation of some quantities related 
to the corresponding  symbolic systems. It was pointed out by our referee that we have 
defined, in a hidden way, the directional complexity and entropy for topological 
Markov chains. We make it explicit in this section. The notion of rotation sets 
for topological Markov chains was introduced in \cite{Ziemian} following general 
approach of \cite{GM}. In our notations it can be described as follows. We consider 
a topological Markov chain $(\Omega_A,\sigma)$ for which the edges $(i,j)$ are endowed 
with integer weights $k_{i,j}$.  We introduce a function $\phi:\Omega_A\to \Z$ as follows:
given $\omega=(\omega_0,\omega_1,\dots)\in\Omega_A$ let $\phi(\omega)=k_{\omega_0,\omega_1}$.
Then the rotation set $\mathcal{J}$ of $\omega$ is
$$
\mathcal{J}(\omega)=
\overline{lt}_{n\to\infty}\frac{1}{n}\sum_{i=0}^{n-1}\phi(\sigma^i\omega)=
\overline{lt}\frac{1}{n}v(\omega_0,\omega_1,\dots,\omega_{n-1}),
$$ 
where $v$ is the weight, defined in subsection~\ref{subsec_estimate} and $\overline{lt}$
is the upper topological limit. The rotation set of the system $(\Omega_A,\sigma)$ is,
by definition, $\bigcup\limits_{\omega\in\Omega_A}{\mathcal J}(\omega)$. The results of 
\cite{Ziemian} imply that, under some conditions, the rotation set is a closed interval.
Now, there are points $\omega\in\Omega_A$ for which 
$\alpha=\lim\limits_{n\to\infty}\frac{1}{n}\sum\limits_{i=0}^{n-1}\phi(\sigma^i\omega)$ exists,
and for a given $\alpha$ we may define directional complexity and entropy. 
\begin{definition}
\begin{enumerate}
\item The number $C_{n,\alpha,r}=|B_{n,\alpha,r}|$ is called the $\alpha$-directional 
$r$-complexity.
\item The number 
$$
e_{\alpha,r}=\ln\lim_{n\to\infty}\sqrt[n]{C_{n,\alpha,r}}
$$
is called the directional $r$-entropy.
\item The number $e_{\alpha}=\lim\limits_{r\to\infty}e_{\alpha,r}$ is called the 
$\alpha$-directional entropy.
\end{enumerate}
\end{definition}
Let us remind that 
$B_{n,\alpha,r}=\{w\in L^{n}\;|\; \forall j=1,\dots,n-1\; \alpha j-r\leq v(w[:j])\leq \alpha j +r\}$, i.e. we admit only those $n$-cylinders for which the weight of a 
$j$-subcylinder can differ from $\alpha j$ no more than by $\pm r$. It is the direct 
analogy with the definition of the ``window-separated points''. There is another way 
to define the directional entropy which was suggested in Section~\ref{sec_com}.
\begin{definition}
The upper topological entropy of the system $(\Omega_A,\sigma)$ is
$\tilde e_\alpha=\ln\lim\limits_{n\to\infty}\sqrt[n]{L^n_{\fl{\alpha n}}}$.
\end{definition}
\begin{theorem}\label{th_last}
If $\tilde e_{\alpha}$ is a continuous at $\alpha$ and $M(L^n_{\fl{\alpha n}})$ have 
a diagonal entry with exponent $\tilde e_\alpha$, then $e_\alpha=\tilde e_\alpha$.
\end{theorem}
\begin{remark}
The method of calculating of $\tilde e_{\alpha}$ described in Section~\ref{sec_com}
works in this more general situation. 
\end{remark}
\begin{remark}
Theorem~\ref{th_last} leaves the possibility that $e_\alpha\neq \tilde e_\alpha$.
The open question is it really may happens. 
\end{remark}
\section{Concluding remarks} \label{sec_con}

Following ideas of Milnor \cite{M86,M88} and also \cite{AZ,ACFM,AMU,CK} we have introduced
and  studied the directional complexity and entropy for dynamical systems generated by 
degree one maps of the circle. In particular, we have considered the maps that admit a Markov 
partition and have positive topological entropy. For them  we have reduced the calculation of the 
$(\epsilon,n)$-complexity on a set of initial points having a prescribed rotation number 
to that of symbolic complexity of admissible cylinders of a topological Markov chain 
(TMC). The admissibility of the cylinders is constructively determined by the rotation number. 
To calculate the symbolic complexity we have used a combinatorial machinery developed in
\cite{P1,P2} adjusted to our situation. As a result we have obtained exact formulas for 
the directional entropy corresponding to every rotation number. Using these formulas 
we have shown that the directional entropy coincides with the measure-theoretic entropy related
to a Markov measure (different for different direction). In particular, we have proved that 
the measure of maximal entropy determines the direction in which the directional entropy
equals the topological entropy of the original dynamical system and, also, we have found an exact
formula for this direction. 

{\bf Acknowledgments.} This work was supported by grant 14-41-00044 of RSF 
at the Lobachevsky University of Nizhny
Novgorod. The authors thank the referee for useful suggestions.

\end{document}